\documentclass[11pt]{article}
\usepackage{a4}
\usepackage{amsfonts,exscale}
\usepackage{mleftright}

\usepackage{srcltx,hyperref}
\usepackage{cite}

\usepackage{color}

\long\def \forget#1{}

\forget{
\pdfoutput=1

\usepackage[pdftex]{hyperref}
\hypersetup{%
pdftitle ={Large sieve inequality},
pdfauthor = { Stephan Baier, Rajneesh Kumar singh},
pdfcreator = {pdfLatex},
pdfproducer={pdfLatex with hyperref}
}
}

\usepackage{amsmath}
\usepackage{amssymb}
\usepackage{amscd}
\usepackage{textcomp}
\usepackage{amsthm}
\theoremstyle{plain}
\newtheorem{Lemma}{Lemma}[section]
\newtheorem{Theorem}[Lemma]{Theorem}

\theoremstyle{definition}

\newcommand{\DS}{\displaystyle}

\let\setminus\smallsetminus

\newcommand{\es}{\enspace}

\newcommand{\dpl}{{\mathchoice{\mbox{\rm (\hspace{-0.15em}(}}
{\mbox{\rm (\hspace{-0.15em}(}}
{\mbox{\scriptsize\rm (\hspace{-0.15em}(}}
{\mbox{\tiny\rm (\hspace{-0.15em}(}}}}
\newcommand{\dpr}{{\mathchoice{\mbox{\rm )\hspace{-0.15em})}}
{\mbox{\rm )\hspace{-0.15em})}}
{\mbox{\scriptsize\rm )\hspace{-0.15em})}}
{\mbox{\tiny\rm )\hspace{-0.15em})}}}}
\newcommand{\invlim}[1][]{\ifthenelse{\equal{#1}{}}
{\DS \lim_{\longleftarrow}}
{\DS \lim_{\underset{#1}{\longleftarrow}}}
}
\newcommand{\dirlim}[1][]{\ifthenelse{\equal{#1}{}}
{\DS \lim_{\longrightarrow}}
{\DS \lim_{\underset{#1}{\longrightarrow}}}
}

\newcommand{\BOne} {{\mathchoice{\hbox{\rm1\kern-2.7pt l\kern.9pt}}
{\hbox{\rm1\kern-2.7pt l\kern.9pt}}
{\hbox{\scriptsize\rm1\kern-2.3pt l\kern.4pt}}
{\hbox{\scriptsize\rm1\kern-2.4pt l\kern.5pt}}}}

\newcommand{\BC}{{\mathbb{C}}}

\newcommand{\BF}{{\mathbb{F}}}

\newcommand{\BN}{{\mathbb{N}}}

\newcommand{\BT}{{\mathbb{T}}}

\newcommand{\sK}{{\mathscr{K}}}

\usepackage{mathrsfs}

\newcommand{\FD}{{\mathfrak{D}}}

\newcommand{\Fa}{{\mathfrak{a}}}

\DeclareMathOperator{\Norm}{Norm}

\DeclareMathOperator{\Tr}{Tr}
\DeclareMathOperator{\vol}{vol}

\begin{document}
\author{ Stephan Baier, Rajneesh Kumar Singh}
\title{The large sieve inequality with square moduli for quadratic extensions of function fields}
\maketitle
\begin{abstract}
In this paper, we establish a version of the large sieve inequality with square moduli for imaginary quadratic extensions 
of rational function fields of odd characteristics.  
\end{abstract}
\tableofcontents
\bigskip

\section{Function fields}
In this section, we introduce some basic notations and provide facts on rational function fields which will be used throughout this paper. 
We follow \cite{Hsu} but in addition introduce the notation $k$ for $\mathbb{F}_q(t)$ and $A$ for $\mathbb{F}_q[t]$ which will be useful when
considering quadratic extensions. 

By $\BF_q$ we denote a field with $q$ elements of characteristic $p$. Let $\Tr : \BF_q \to \BF_p$ be the trace map and 
$\BF_q(t)$ the function field in one variable over $\BF_q$.
Let $k_\infty$ be the completion of $k= \BF_q(t)$ at $\infty$ (i.e. $k_{\infty}=\BF_q\dpl 1/t\dpr$). The absolute value 
$|\cdot |_\infty$ on $k_\infty$ is defined by 
\begin{align*}
\mathrel \bigg | \sum_{i= -\infty}^n a_i t^i\bigg| = q^n, \es \text{if} \es 0 \neq a_n \in \BF_q.
\end{align*}
The non-trivial additive character $E: \BF_q \to \BC^\times$ is defined by 
\begin{align*}
E(x) = \exp \bigg\{\frac{2\pi i}{p} \Tr(x)\bigg\},
\end{align*}
where the map $e: k_\infty \to \BC^\times$ is defined by 
\begin{align*}
e\bigg(\sum_{i= -\infty}^n a_i t^i\bigg) = E(a_{-1}).
\end{align*}
This map $e$ is also a non-trivial additive character of $k_\infty$.

Given $f = (f_1, f_2,\cdots, f_n)\in k^n_\infty$, we define the additive character $\Psi_f:  k^n_\infty \to \BC^\times$ as 
\begin{align*}
\Psi_f((g_1,g_2,\cdots,g_n))&= e(f_1g_1+f_2g_2+\cdots,f_ng_n)\\
&=\prod_{i=1}^ne(f_ig_i)
\end{align*}
for any $(g_1,g_2,\cdots,g_n)\in k^n_\infty$.
Suppose $f = (f_1, f_2,\cdots, f_n)\in k^n_\infty$. We define a norm $|\cdot |_\infty$ on $k^n_\infty$ by 
\begin{align*}
|f |_\infty = \sup\{|f_1 |_\infty,\ |f_2 |_\infty, \cdots,|f_n |_\infty\}.
\end{align*}
Given any $N\in \mathbb{R}$, the $N$-ball $B_n(f,N)$ is defined by 
\begin{align*}
B_n(f,N) = \{g\in  k^n_\infty \ : \ |g-f|_\infty \leq q^N  \}.
\end{align*}
We view $A^n : = \BF_q[t]^n\subset k^n_\infty$ as a lattice of rank $n$ over $A = \BF_q[t]$ 
and define the $n$-torus to be $\BT_n = k^n_\infty/A^n $. The induced metric on $\BT_n$ is given by 
\begin{align*}
\|h\|_\infty = \inf_{h'\in h +A^n}|h'|_\infty. 
\end{align*}
Note that 
$\BT_n$ is a compact Hausdorff space and for all $h\in \BT_n,\ \|h\|_\infty\leq 1/q$.

\section{Imaginary quadratic extensions}
In this section, we introduce imaginary quadratic extensions of $\BF_q(t)$ and continue with notations and basic facts. For more details on algebraic extensions of function fields, see \cite{Ros}.

First, let $K$ be a general finite separable extension of $k = \BF_q(t)$ of degree $n$. We denote the integral closure of $A=\BF_q[t]$ in $K$ by $\mathbb{A}$ and 
let $(\theta_1,...,\theta_n)$ be an integral basis of $K$, so that every integer $\xi\in \mathbb{A}$ is representable uniquely as 
$$
\xi=f_1\theta_1+...+f_n\theta_n,
$$
where $f_1,...,f_n$ belong to $A$. For any integral ideal $\mathfrak{a}$ of $K$, let $\sigma(\xi)$ be an additive character modulo 
$\mathfrak{a}$. Such a character is called proper if it is not an additive character modulo an ideal $\mathfrak{b}$ which divides 
$\mathfrak{a}$ properly.   
If $\xi = \xi'\bmod{\Fa}$ and $\eta \in \mathbb{A}/ \FD_K \Fa$, where $\FD_K$  is the different of the field $K$, 
then 
$$
\xi \eta - \xi'\eta \in \FD_K^{-1}= \{a\in K \mid \mathbf{Tr}(a)\in \BF_q[t]\},
$$
so that
$$
\mathbf{Tr}(\xi\eta) = \mathbf{Tr}(\xi'\eta) \bmod{A}
$$
and therefore 
$$
e(\mathbf{Tr}(\xi\eta)) = e (\mathbf{Tr}(\xi'\eta)).
$$
Here $\mathbf{Tr}$ denotes the trace map $\mathbf{Tr}: K \to k$.

Now we restrict ourselves to imaginary quadratic extensions. We further assume that $q$ is odd. 
Let $\alpha\in A$ be a square-free polynomial which is of odd degree or whose leading coefficient is not a square in $\mathbb{F}_q$ and set
$$
K := k\left(\sqrt{\alpha}\right).
$$
Then we say that the field extension $K:k$ is imaginary quadratic. Let $\mathbb{A}$ be the 
integral closure of $A$ in $K$. We know that $\mathbb{A}$ is equal to 
$A + A\sqrt{\alpha}$ and $\FD_K$ equals $\left(2\sqrt{\alpha}\right)$ (see \cite[page 248]{Ros}). In general, $\mathbb{A}$ is not necessarily a  principal ideal domain, but in the sequel, 
we will restrict ourselves to principal ideals in $\mathbb{A}$. 
The proper additive characters for the ideal $\Fa=(f)$ take the form 
$$
\sigma(\xi)=e\left(\mathbf{Tr}\left(\frac{\xi r}{2\sqrt{\alpha}f}\right)\right) \ \text{for some}\ r \ \text{with} \ (r,f)=1.
$$
To see this, one notes that since $(r,f)=1$, these are indeed proper characters for $(f)$ and, moreover, there are no other proper characters since the number of
them equals, by duality, the order of the group of units in the quotient ring $\mathbb{A}/(f)$. 
Norm and trace of an element $f = a + b\sqrt{\alpha}\in K$ are 
given as 
$$
\mbox{Norm}(f) =  a^2- b^2 \alpha \quad \mbox{and} \quad \mathbf{Tr}(f)=2a.
$$

We fix a place in $K$ above $\infty$ and denote it by $\infty$ as well. 
Let $K_\infty$ be the completion of $K$ at this place. The absolute value on $k_{\infty}$ extends uniquely to an absolute value on $K_{\infty}$ which we also denote by $|.|_{\infty}$. 
Henceforth, we write
$$
\mathbf{N}(f):=|f|_{\infty}^2
$$
for $f\in \mathbb{K}_{\infty}$.
We note that
$$
\mathbf{N}(f):= q^{\deg(a^2-b^2\alpha)} \quad \mbox{if } f=a+b\sqrt{\alpha}\in \mathbb{A}.
$$
For any given  integer $N\in \mathbb{R}$, we write  
\begin{align*}
\mathbf{B}(f,N) : = \{g\in  K_\infty \ : \ \mathbf{N}{(g-f)} \leq q^N  \}=\{g\in  K_\infty \ : \ |g-f|_{\infty} \leq q^{N/2}  \},
\end{align*}
which may be viewed as the $N/2$-ball with center $f$ in $K_{\infty}$. 
We view $\mathbb{A} \subset K_\infty$ as a lattice 
and define the corresponding torus to be $\mathbf{T} = K_\infty/\mathbb{A}$. The induced metric on $\mathbf{T}$ is given by 
\begin{equation*} 
\|h\|_{\mathbf{T},\infty} := \inf_{h'\sim h}|h'|_\infty.
\end{equation*}
We set
\begin{equation} \label{newtorus}
\mathbf{N}_{\mathbf{T}}(h):=\|h\|_{\mathbf{T},\infty}^2.
\end{equation}

\section{Main results}
The goal of this paper is to establish large sieve inequalities for $K=k(\sqrt{\alpha})$, extending earlier work for rational function fields 
$k$ in \cite{BSi3} and \cite{Hsu}. We first prove the following analogue of the classical large sieve inequality 
for imaginary quadratic extensions of function fields. 

\begin{Theorem}\label{fullmodqe} Let $Q, N \in \mathbb{N}$ and $(a_g)_{g \in \mathbb{A}}$ be any sequence of complex numbers. Then 
\begin{equation*}
\begin{split}
& \sum\limits_{\substack{f\in \mathbb{A}\setminus\{0\}\\ \mathbf{N}(f)= q^Q}} \ \sum\limits_{\substack{r \bmod{f}\\ (r,f)=1}} 
\left|\sum \limits_{g \in \mathbf{B}(0, N)\cap \mathbb{A}}  a_g \cdot e\left(\mathbf{Tr}\left(\frac{gr}{2\sqrt{\alpha}f}\right)\right)\right|^2 \\
\ll & \left(q^{2Q}+q^N\right) \sum \limits_{g \in \mathbf{B}(0, N)\cap \mathbb{A}}   |a_g|^2,
\end{split}
\end{equation*}
where the implied constant depends only on $q$ and $\deg(\alpha)$.
\end{Theorem}

Then we restrict ourselves to square moduli. We imitate the procedure in \cite{BBa}, where a version of the large sieve with square moduli was obtained for
$\mathbb{Q}(i)$. This procedure itself is based on L. Zhao's work \cite{Zh1} on the large sieve with square moduli for $\mathbb{Q}$. 
Here, we shall establish the following.

\begin{Theorem}\label{squaremodqe} Let $Q, N \in \mathbb{N}$ and $(a_g)_{g \in \mathbb{A}}$ be any sequence of complex numbers. 
Assume that $\mbox{char}(K)=p>2$. Then 
\begin{equation} \label{in}
\begin{split}
& \sum\limits_{\substack{f\in \mathbb{A}\setminus\{0\}\\ \mathbf{N}(f)= q^Q}} \ \sum\limits_{\substack{r \bmod{f^2}\\ (r,f)=1}} 
\left|\sum \limits_{g \in \mathbf{B}(0, N)\cap \mathbb{A}}  a_g \cdot e\left(\mathbf{Tr}\left(\frac{gr}{2\sqrt{\alpha}f^2}\right)\right)\right|^2 
\\ \ll & \left(q^{3Q}+2^{(N-Q)/2}\left(q^{Q/2+N}+q^{2Q+N/2}\right)\right)
\sum \limits_{g \in \mathbf{B}(0, N)\cap \mathbb{A}}   |a_g|^2,
\end{split}
\end{equation}
where the implied constant depends only on $q$ and $\deg(\alpha)$.
\end{Theorem}

For comparison, the large sieve with square moduli for $\mathbb{Q}(i)$ established in \cite{Zh1} is the following inequality.

\begin{Theorem} \label{powermodZi} Let $Q,N\ge 1$ and $(a_n)_{n\in \mathbb{Z}[i]}$ be any sequence of complex numbers. Then
\begin{equation*}
\begin{split}
& \sum\limits_{\substack{q\in \mathbb{Z}[i]\setminus\{0\}\\ \mathbf{N}(q)\le Q}} \
\sum\limits_{\substack{r \bmod{q^2}\\ (r,q)=1}} \left|\sum\limits_{\substack{n\in \mathbb{Z}[i]\\ \mathbf{N}(n)\le N}}  
a_n \cdot e\left(\Re\left(\frac{nr}{q^2}\right)\right)\right|^2\\ \ll_{m,\varepsilon} & 
(QN)^{\varepsilon}\left(Q^{3}+NQ^{1/2}+N^{1/2}Q^{2}
\right)
\sum\limits_{\substack{n\in \mathbb{Z}[i]\\ \mathbf{N}(n)\le N}} |a_n|^2.
\end{split}
\end{equation*}
\end{Theorem}

The term $q^Q$ in Theorem \ref{squaremodqe} corresponds to the variable $Q$ in Theorem \ref{powermodZi}, $q^N$ corresponds to $N$, and the term $2^{(N-Q)/2}$ 
to the factor $(QN)^{\varepsilon}$.\\ \\
{\bf Acknowledgements:} We would like to thank the referee for his valuable comments.\\ \\
\section{Preliminaries}
In this section, we provide the tools needed in this paper. First, we 
shall use the following general large sieve inequality for higher dimensions due to the authors.  
In fact we will need this result only for the two-dimensional case. 

\begin{Theorem} \label{basiclargesieve}
Let $R, N \in \BN, \ X_1,\cdots, X_R \in k_\infty^n$ and $(a_g)_{g \in A^n}$ be any sequence of complex numbers. Then 

\begin{align}\label{lshighdim}
\sum\limits_{i=1}^{R}\mathrel \bigg|\sum\limits_{g \in B_n(0,N)\cap A^n}a_g \ e(g\cdot X_i)\bigg|^2 \ll  \es q^{n(N+1)}\cdot \sK\cdot
\sum\limits_{g \in B_n(0,N)\cap A^n}|a_g|^2, 
\end{align}
where 
$$ 
\sK: = \max_{1\leq i_1\leq R}\#\Big\{1\leq i_2\leq R \ : \|X_{i_1}-X_{i_2}\|_{\infty}\leq  q^{-(N+2)}\Big\}
$$
and $\|\cdot\|_{\infty}$ is the induced distance on the torus $\BT_n =k_\infty^n/A^n$.
\end{Theorem}

\begin{proof}
This is found in \cite[section 5]{BSi2}.
\end{proof}

{\bf Remark:} Recently, in \cite{BSi1}, we pointed out an error in \cite{BSi2}. This error, however, does neither affect the above result nor any of 
the results stated below. \\

We further use the Poisson summation formula for lattices in $k_{\infty}^n$. 

\begin{Theorem}[Poisson Summation Formula] \label{Thm2} Let $\Lambda$ be a complete lattice in $k^n_\infty$ and let 
\begin{align*}
\Lambda'= \Big\{g\in k^n_\infty \mathrel \Big| f\cdot g \in A \text{ for all } f\in \Lambda \Big\}
\end{align*}
be the lattice dual to $\Lambda$.  Let $\Phi: k^n_\infty \to \BC$ be a function such that 
$$
\sum_{y\in \Lambda}|\Phi(h+y)|
$$ 
is uniformly convergent on compact subsets of $k^n_\infty$ and 
$$
\sum_{x\in \Lambda'}|\hat{\Phi}(x)|
$$ 
is convergent. Then 
\begin{align*}
\sum_{y\in \Lambda}\Phi(y) = \frac{1}{\vol\left(k_{\infty}^n/\Lambda\right)}\sum_{x\in \Lambda'}\hat{\Phi}(x),
\end{align*}  
where $\vol\left(k_{\infty}^n/\Lambda\right)$ is the volume of a fundamental mesh of $\Lambda$.
\end{Theorem}

\begin{proof}
This is \cite[Theorem 3.2.]{BSi2}, where we note that the condition $B(f,g)=1$ in the definition of the dual lattice $\Lambda'$ in \cite[Theorem 3.2.]{BSi2} is 
equivalent to the condition $f\cdot g\in A$ above. 
\end{proof}

A simple consequence of this is the following extended version of the Poisson summation formula. 

\begin{Theorem}[Poisson Summation Formula] \label{Thm3}  Under the conditions of Theorem \ref{Thm2}, we have 
\begin{align} \label{extens}
\sum_{y\in h+\Lambda}\Phi\left(\frac{y}{t^L}\right) = \frac{q^{nL}}{\vol\left(k_{\infty}^n/\Lambda\right)}
\sum_{x\in \Lambda'} e(h\cdot x) \hat{\Phi}\left(t^Lx\right)
\end{align}  
for all $h\in k^n_\infty$ and $L\in \mathbb{Z}$.
\end{Theorem}

\begin{proof}
This follows from Theorem \ref{Thm2} using the linear change of variables $z=y-h$. Since $y$ ranges over $h+\Lambda$, $z$ ranges over the lattice $\Lambda$. 
Now applying Theorem \ref{Thm2} above with
$y$ replaced by $z$, we
obtain \eqref{extens}, where the term $e(h\cdot x)$ on the right-hand side comes out of the calculation of the Fourier integral after this change of variables. 
\end{proof}

We will work with the characteristic function of the unit ball $B(0,1)$. For our application of the Poisson summation formula, we need the 
Fourier transform of this function, which was also calculated in \cite{BSi2}.

\begin{Lemma} \label{fourier}
Let $\Phi_1: k^n_\infty \to \BC$ be defined as 
\begin{equation} \label{Phidef}
\Phi_1(y): = 
\begin{cases}
1, &  \text{if} \ |y|_\infty \leq 1, \\
0, & \text{otherwise.}
\end{cases} 
\end{equation}
Then the Fourier transform $\hat\Phi_1$ of $\Phi_1$ satisfies 
\begin{align*}
\hat \Phi_1(x) = q^n \Phi_1\left(t^2x\right)= \begin{cases}
q^n, &  \text{if} \ |x|_\infty \leq q^{-2}, \\
0, & \text{otherwise}.
\end{cases}
\end{align*}
\end{Lemma}

\begin{proof}
This is found in \cite[section 5]{BSi2}. 
\end{proof}

\section{Application of the two-dimensional large sieve}
Now we return to the large sieve for $K = k\left(\sqrt{\alpha}\right)$.  We begin with restricting the moduli $f$ to subsets $\mathcal{S}$ of
$\mathbb{A}$. Further, we write 
$$
\mathcal{S}(Q):=\{ f\in \mathcal{S} \ :\ \mathbf{N}(f) = q^Q\}.
$$
Thus, we are interested in estimating the quantity
\begin{equation} \label{Tdef}
T : = \sum\limits_{f\in \mathcal{S}(Q)} \ \sum\limits_{\substack{r \bmod{f}\\ (r,f)=1}} 
\left|\sum \limits_{g \in \mathbf{B}(0, N)\cap \mathbb{A}}  a_g \cdot e\left(\mbox{\rm Tr}\left(\frac{gr}{2\sqrt{\alpha}f}\right)\right)\right|^2 .
\end{equation}
We shall later confine ourselves to the cases when $\mathcal{S}$ equals $\mathbb{A}$ or the set of 
squares.  

In the following, we write
$$
g = s + w\sqrt{\alpha} , \  r =  x + y \sqrt{\alpha}, \  f =  u + v \sqrt{\alpha}, 
$$
where $s, w, x, y, u , v \in A$. Then  
\begin{equation*}
\frac{gr}{2\sqrt{\alpha}f} 
=  \frac{ s(xu-yv \alpha) +w \alpha(yu-xv) + (w(xu-yv \alpha)+ s(yu-xv))\sqrt{\alpha}}{2 \sqrt{\alpha}\cdot \mbox{Norm}(f) }.
\end{equation*}
It follows that
\begin{equation*}
\begin{split}
\mathbf{Tr}\left(\frac{gr}{2\sqrt{\alpha}f}\right) = & \frac{ w(xu-yv \alpha)+ s(yu-xv)}{  \mbox{Norm}(f) }\\ 
= & \left( \frac{   yu-xv}{ \mbox{Norm}(f) }, 
\frac{ xu-yv \alpha}{  \mbox{Norm}(f) }\right)\cdot (s,w).
\end{split}
\end{equation*}
Hence, we can re-write $T$ in the form
\begin{equation*}
T = \sum\limits_{f\in \mathcal{S}(Q)} \ \sum\limits_{\substack{r \bmod{f}\\ (r,f)=1}} \left|\sum \limits_{g \in \mathbf{B}(0, N)\cap \mathbb{A}}  
a_g \cdot e\left(\left( \frac{   yu-xv}{ \mbox{Norm}(f) }, \frac{ xu-yv \alpha}{ \mbox{Norm}(f) }\right)\cdot (s,w)\right)\right|^2.
\end{equation*}
The condition  
$$
g = s + w \sqrt{\alpha} \in \mathbf{B}(0, N)\cap \mathbb{A}
$$
is equivalent to  
$\deg(s^2-w^2 \alpha)\leq N,$ which in turn is 
equivalent to 
$$
\max\{\deg s^2, \deg(w^2 \alpha)\} \leq N
$$
since in the difference $s^2-w^2\alpha$ above, the highest terms of $s^2$ and $w^2\alpha$ don't cancel out because we assume 
that $\alpha$ is of odd degree or the leading coefficient of $\alpha$ is a non-square in 
$\mathbb{F}_q$. Further, the above inequality implies that
$(s,w) \in B_2(0,N/2)\cap A^2$. It follows that
\begin{equation*}
\begin{split}
T = & \sum\limits_{f\in \mathcal{S}(Q)} \ \sum\limits_{\substack{r \bmod{f}\\ (r,f)=1}}\\ & 
\left|\sum\limits_{(s,w) \in B_2(0,N/2)\cap A^2}  a_{g} \cdot e\left(\left( \frac{yu-xv}{ \mbox{Norm}(f) }, \frac{ xu-yv \alpha}{ 
\mbox{Norm}(f) }\right)\cdot (s,w)\right)\right|^2, 
\end{split}
\end{equation*}
where we set $a_g:=0$ if $2\deg(w)+\deg(\alpha)>N$. 

Employing Theorem \ref{basiclargesieve} with $n=2$, we deduce that
\begin{equation} \label{Tbound}
T\ll q^{N+2}\cdot \sK\cdot\sum\limits_{g \in  \mathbf{B}(0, N)\cap \mathbb{A}}|a_g|^2,
\end{equation}
where 
\begin{equation} \label{sK}
\begin{split}
\sK: = & \max\limits_{\substack{r_1 \in \mathbb{A}, \ f_1\in \mathcal{S}(Q) \\ (r_1,f_1)=1} }
\#\bigg\{(r_2,f_2) \ :\ f_2\in \mathcal{S}(Q), \ (r_2,f_2)=1, \\ \\ & 
\mathrel\bigg\| 
\left(\frac{   y_1u_1-x_1v_1}{ \mbox{\rm Norm}(f_1) }, \frac{ x_1u_1-y_1v_1 \alpha}{\mbox{\rm Norm}(f_1) }\right)-
\left(\frac{   y_2u_2-x_2v_2}{ \mbox{\rm Norm}(f_2) }, 
\frac{ x_2u_2-y_2v_2 \alpha}{\mbox{\rm Norm}(f_2) }\right) 
\mathrel\bigg\|_{\infty}\\
& \leq  q^{-(N/2+2)}\bigg\}
\end{split}
\end{equation}
with the conventions that
$$  
f_j =  u_j + v_j \sqrt{\alpha} , \hspace{.2cm} r_j = x_j + y_j \sqrt{\alpha} 
$$
for $j=1,2$. Thus, the remaining task is to bound the quantity $\sK$. 

\section{General bounds for $\sK$}
We observe that in the definition of $\sK$ in \eqref{sK}, 
the distance $||.||_{\infty}$ on $\mathbb{T}_2$ can be replaced by the distance $|.|_{\infty}$ on 
$k_{\infty}^2$. 
Moreover,
$$
\frac{r_j}{f_j} = \frac{ x_ju_j-y_jv_j \alpha}{ \Norm(f_j) } + \frac{   (y_ju_j-x_jv_j) }{ \Norm(f_j) } \cdot \sqrt{\alpha}
$$
for $j=1,2$. It follows that 
\begin{align*}
\sK \le 
\max\limits_{\substack{r_1 \in \mathbb{A}, \ f_1\in \mathcal{S}(Q) \\ (r_1,f_1)=1} }
\#\bigg\{(r_2,f_2) \ :\ f_2\in \mathcal{S}(Q), \ (r_2,f_2)=1, \
\mathbf{N}\left(\frac{r_1}{f_1}-\frac{r_2}{f_2}\right)\leq q^{\ell -( N+4) } \bigg\},
\end{align*}
where
$$
\ell := \mbox{deg}(\alpha).
$$
Moreover, since $\mathbf{N}(f_j) = q^Q$ for $j=1,2$, the inequality
$$
 \mathbf{N}\left(\frac{r_1}{f_1}-\frac{r_2}{f_2}\right)\leq q^{\ell -( N+4) }
$$
above is equivalent to
$$
\deg\left(\Norm\left(r_1f_2-r_2f_1\right)\right)\leq 2Q+ \ell -(N+4).
$$
Therefore,
\begin{equation*} 
\begin{split}
\sK \le & \max\limits_{\substack{r_1 \in \mathbb{A}, \ f_1 \in \mathcal{S}(Q)  
\\ (r_1,f_1)=1} }\#\bigg\{(r_2,f_2) \ :\ f_2\in \mathcal{S}(Q), \ (r_2,f_2)=1, \
\mathbf{D}\left(r_1f_2-r_2f_1\right)\leq M  \bigg\},
\end{split}
\end{equation*}
where
\begin{equation} \label{Mdef}
M:=2Q+\ell -( N+4),
\end{equation}
and for $c=a+b\sqrt{\alpha}\in A$, we set
$$
\mathbf{D}(c):=\deg(\Norm(c)).
$$
We write the above in the form
\begin{equation} \label{sKbound}
\begin{split}
\sK \le & \max\limits_{\substack{r_1 \in \mathbb{A}, \ f_1 \in \mathcal{S}(Q) \\ (r_1,f_1)=1} }
\sum\limits_{\substack{h\in \mathbb{A}\\ \mathbf{D}(h)\le M}}\
\sum\limits_{\substack{f_2\in \mathcal{S}(Q)\\ r_1f_2\equiv h \bmod{f_1}}} 1\\
\le & \max\limits_{\substack{r_1 \in \mathbb{A}, \ f_1 \in \mathcal{S}(Q) \\ (r_1,f_1)=1} }\
\sum\limits_{f_2\in \mathcal{S}(\le Q)} \
\sum\limits_{\substack{h\in \mathbb{A}\\ \mathbf{D}(h)\le M\\ h \equiv r_1f_2 \bmod{f_1}}} 1,
\end{split}
\end{equation}
where in the last line, we have extended the summation over $f_2$ to the set
$$
\mathcal{S}(\le Q):=\{ f\in \mathcal{S} \ :\ \mathbf{N}(f) \le q^Q\}=\{ f\in \mathcal{S} \ :\ \mathbf{D}(f) \le Q\},
$$
which will later turn out to be convenient.

\section{Poisson summation}
If $M<0$, then the innermost sum over $h$ in the second line of \eqref{sKbound} satisfies
\begin{equation*}
\sum\limits_{\substack{h\in \mathbb{A}\\ \mathbf{D}(h)\le M\\ h \equiv r_1f_2 \bmod{f_1}}} 1 = 
\begin{cases}
1 & \mbox{ if } f_2\equiv 0 \bmod{f_1}\\
0 & \mbox{ otherwise}
\end{cases}
\end{equation*}
since $(r_1,f_1)=1$. To see this, note that if $h\not=0$, then $\mathbf{D}(h)\ge 0$. 
Since $\mathbf{D}(f_1)=Q$, the number of polynomials $f_2$ such that $\mathbf{D}(f_2)\le Q$ and $f_2\equiv 0 \bmod{f_1}$ is bounded by
$q$. It follows from \eqref{sKbound} that
\begin{equation} \label{M<0}
\sK \le q \quad \mbox{if } M<0.
\end{equation}

If $M\ge 0$, we transform the sum over $h$ in question using Poisson summation. Any multiple of 
\begin{equation} \label{f1def}
f_1=u_1+v_1\sqrt{\alpha}
\end{equation}
by an element of $\mathbb{A}$ takes the form
$$
(p_1+q_1\sqrt{\alpha})(u_1+v_1\sqrt{\alpha})=(p_1u_1+q_1v_1\alpha)+(p_1v_1+q_1u_1)\sqrt{\alpha}.
$$
This corresponds to a lattice of the form
$$
\Lambda=\left\{p_1 \binom{u_1}{v_1} + q_1 \binom{v_1\alpha}{u_1}\ :\ (p_1,q_1)\in A^2\right\}
$$
in the sense that the multiples of $f_1$ are precisely of the form $a+b\sqrt{\alpha}$, where 
$$
\binom{a}{b}\in \Lambda.
$$
Let 
\begin{equation} \label{r1f2def}
r_1f_2=c+d\sqrt{\alpha}.
\end{equation}
Then it follows that
\begin{equation} \label{less}
\begin{split}
\sum\limits_{\substack{h\in \mathbb{A}\\ \mathbf{D}(h)\le M\\ h \equiv r_1f_2 \bmod{f_1}}} 1
\le \sum\limits_{\substack{y\in (c,d)+\Lambda\\ \max\{\deg(c),\deg(d)\}\le [M/2]}} 1
= \sum\limits_{y\in (c,d)+\Lambda} \Phi_1\left(\frac{y}{t^{[M/2]}}\right), 
\end{split}
\end{equation}
where $\Phi$ is defined as in \eqref{Phidef}. Using the Poisson summation formula, Theorem \ref{Thm3}, and Lemma \ref{fourier}, we get
\begin{equation*}
 \sum\limits_{y\in (c,d)+\Lambda} \Phi_1\left(\frac{y}{t^{[M/2]}}\right) = 
 \frac{q^{2[M/2]+2}}{\vol\left(k_{\infty}^2/\Lambda\right)}
\sum_{\substack{x\in \Lambda'\\ |x|_{\infty}\le t^{-([M/2]+2)}}} e((c,d)\cdot x).
\end{equation*}
In our case, the dual lattice takes the form
$$
\Lambda'= \frac{1}{\mbox{Norm}(f_1)} \cdot \hat\Lambda:=
\frac{1}{\mbox{Norm}(f_1)} \cdot \left\{p_1 \binom{u_1}{-v_1\alpha} + q_1 \binom{-v_1}{u_1}\ :\ (p_1,q_1)\in A^2\right\},
$$
and the volume of a fundamental mesh of $\Lambda$ is 
$$
\vol\left(k_{\infty}^2/\Lambda\right)=\mathbf{N}(f_1)=q^Q.
$$
Taking into account that $\mathbf{N}(f_1)=q^Q$, it follows that 
\begin{equation*}
\begin{split}
& \sum\limits_{y\in (c,d)+\Lambda} \Phi_1\left(\frac{y}{t^{[M/2]}}\right)\\ 
= & q^{2[M/2]+2-Q}
\sum_{\substack{\hat{x}\in \hat \Lambda\\ |\hat{x}|_{\infty}\le q^{Q-([M/2]+2)}}} e\left((c,d)\cdot \frac{\hat{x}}{\mbox{Norm}(f_1)}\right)\\
= & q^{2[M/2]+2-Q}
\sum_{\substack{(p_1,q_1)\in A^2\\ \deg(p_1u_1-q_1v_1)\le Q-([M/2]+2)\\ 
\deg(-p_1v_1\alpha+q_1u_1)\le Q-([M/2]+2)}} 
e\left(\frac{(cp_1u_1-dp_1v_1\alpha)+(-cq_1v_1+dq_1u_1)}{\mbox{Norm}(f_1)}\right)\\
= & q^{2[M/2]+2-Q}
\sum_{\substack{q_1+p_1\sqrt{\alpha}\in \mathbb{A}\\ \deg(p_1u_1-q_1v_1)\le Q-([M/2]+2) \\ \deg(-p_1v_1\alpha+q_1u_1)\le Q-([M/2]+2)}}
e\left({\bf Tr}\left(\frac{(q_1+p_1\sqrt{\alpha})(c+d\sqrt{\alpha})(u_1-v_1\sqrt{\alpha})}{2\mbox{Norm}(f_1)\sqrt{\alpha}}\right)
\right).
\end{split}
\end{equation*}
Recalling \eqref{f1def} and \eqref{r1f2def}, and writing $j=q_1+p_1\sqrt{\alpha}$,
$$
\Re(a+b\sqrt{\alpha})=a,\ \Im(a+b\sqrt{\alpha})=b,\ \overline{a+b\sqrt{\alpha}}=a-b\sqrt{\alpha} \quad \mbox{for } a+b\sqrt{\alpha}\in 
\mathbb{A}
$$
and 
\begin{equation} \label{M1M2def}
M_1:=2[M/2]+2-Q \quad \mbox{and}\quad M_2:=Q-([M/2]+2),
\end{equation}
we arrive at
\begin{equation} \label{arrive}
\begin{split}
\sum\limits_{y\in (c,d)+\Lambda} \Phi_1\left(\frac{y}{t^{[M/2]}}\right)= & q^{M_1}
\sum_{\substack{j\in \mathbb{A}\\ \deg(\Re(j\overline{f_1}))\le M_2 \\ \deg(\Im(j\overline{f_1}))\le M_2}}
e\left({\bf Tr}\left(\frac{jr_1f_2 \overline{f_1} }{2\mbox{Norm}(f_1)\sqrt{\alpha}}\right)
\right)\\
= & q^{M_1}
\sum_{\substack{j\in \mathbb{A}\\ \deg(\Re(j\overline{f_1}))\le M_2 \\ \deg(\Im(j\overline{f_1}))\le M_2}}
e\left({\bf Tr}\left(\frac{jr_1f_2}{2f_1\sqrt{\alpha}}\right)
\right).
\end{split}
\end{equation}
Combining \eqref{sKbound}, \eqref{less} and \eqref{arrive}, exchanging summations, and using \eqref{Mdef} and \eqref{M1M2def}, we get
\begin{equation} \label{sKnewbound}
\begin{split}
\sK \le & q^{M_1}\max\limits_{\substack{r_1 \in \mathbb{A}, \ f_1 \in \mathcal{S}(Q) \\ (r_1,f_1)=1} }\
\sum\limits_{f_2\in \mathcal{S}(\le Q)} \sum_{\substack{j\in \mathbb{A}\\ \deg(\Re(j\overline{f_1}))\le M_2 \\ \deg(\Im(j\overline{f_1}))\le M_2}}
e\left({\bf Tr}\left(\frac{jr_1f_2}{2f_1\sqrt{\alpha}}\right)\right)\\
\le & q^{M_1}\max\limits_{\substack{r_1 \in \mathbb{A}, \ f_1 \in \mathcal{S}(Q) \\ (r_1,f_1)=1} }\
\sum_{\substack{j\in \mathbb{A}\\ \mathbf{D}(j) \le 2(M_2-Q)+\ell}} \left| 
\sum\limits_{f_2\in \mathcal{S}(\le Q)} 
e\left({\bf Tr}\left(\frac{jr_1f_2}{2f_1\sqrt{\alpha}}\right)\right) \right|\\
\le & q^{Q+\ell-(N+2)}\max\limits_{\substack{r_1 \in \mathbb{A}, \ f_1 \in \mathcal{S}(Q) \\ (r_1,f_1)=1} }\
\sum_{\substack{j\in \mathbb{A}\\ \mathbf{D}(j) \le N+1-Q}} \left| 
\sum\limits_{f_2\in \mathcal{S}(\le Q)} 
e\left({\bf Tr}\left(\frac{jr_1f_2}{2f_1\sqrt{\alpha}}\right)\right) \right|
\end{split}
\end{equation}
if $M\ge 0$.

\section{Linear exponential sums}
Suppose that $X\ge \ell$. In the following, we estimate linear exponential sums of the form
$$
\Sigma(X,h):=\sum\limits_{\substack{x\in \mathbb{A}\\ \mathbf{D}(x)\le X}} 
e\left(\mathbf{Tr}\left(\frac{h\cdot x}{2\sqrt{\alpha}}\right)\right),
$$
where $h\in K$. Let
$$
h=a+b\sqrt{\alpha} \quad \mbox{with } a,b\in k.
$$
Then 
\begin{equation*}
\begin{split}
\Sigma(X,h)= & \sum\limits_{\substack{(u,v)\in \mathbb{A}^2\\ \deg(u)\le [X/2]\\ \deg(v)\le [(X-\ell)/2]}} 
e\left(ub+va\right)\\ = & \left(\sum\limits_{\substack{u\in A\\ \deg(u)\le [X/2]}} e(ub)\right)\cdot 
\left(\sum\limits_{\substack{v\in A\\ \deg(t)\le [(X-\ell)/2]}} e(va)\right)\\
= & \left(\sum\limits_{u\in A} e(ub)\cdot \Phi_1\left(t^{-[X/2]}u\right)\right)\cdot 
\left(\sum\limits_{v\in A} e(va)\cdot \Phi_1\left(t^{-[(X-\ell)/2]}v\right)\right).
\end{split}
\end{equation*}
Suppose that $\hat\Phi=\Phi_1$. Using Lemma \ref{fourier} with $n=1$, we have  
\begin{align*}
\Phi(y)=\hat{\hat \Phi}(-y) = \hat\Phi_1(-y)= q \hat\Phi_1\left(-t^2y\right)= \begin{cases}
q, &  \text{if} \ |y|_\infty \leq q^{-2}, \\
0, & \text{otherwise},
\end{cases}
\end{align*}
where the equation $\Phi(y)=\hat{\hat \Phi}(-y)$ is a consequence of the Fourier inversion formula. 
Hence, using Theorem \ref{Thm3} with $n=1$ and $L=[X/2]$, and noting that 
$
\vol(k_{\infty}/A)=1,
$
we have, denoting by $||.||_{\infty}$ the metric on $\mathbb{T}_1$ induced by $|.|_{\infty}$,
\begin{equation*}
\begin{split}
\sum\limits_{u\in A} e(ub)\cdot \Phi_1\left(t^{-[X/2]}u\right) = &
\sum\limits_{u\in A} e(ub)\cdot \hat\Phi\left(t^{-[X/2]}u\right)\\ = & 
q^{[X/2]} \sum_{y\in b+A}\Phi\left(-t^{[X/2]}y\right)\\ = &
\begin{cases} 
q^{[X/2]+1}  & \mbox{ if } ||b||_{\infty}\le q^{-(2+[X/2])}\\ 
0 & \mbox{ otherwise}
\end{cases}
\end{split}
\end{equation*}
(only one term in the sum above survives) and similarly,
\begin{equation*}
\begin{split}
\sum\limits_{v\in A} e(va)\cdot \Phi_1\left(t^{-[X/2]}v\right) = 
\begin{cases} 
q^{[(X-\ell)/2]+1}  & \mbox{ if } ||a||_{\infty}\le q^{-(2+[(X-\ell)/2])}\\ 
0 & \mbox{ otherwise.}
\end{cases}
\end{split}
\end{equation*}
Recall the definition of $\mathbf{N}_{\mathbf{T}}(h)$ in 
\eqref{newtorus}. We observe that if $\mathbf{N}_{\mathbf{T}}(h)\le q^{\ell-X-1}$, then $||a||_{\infty}\le q^{-(2+[(X-\ell)/2])}$ and $||b||_{\infty}\le q^{-(2+[X/2])}$. We
deduce the following.

\begin{Lemma} \label{linexpsum}
If $X\ge \ell$, then
$$
|\Sigma(X,h)|\le \begin{cases} q^{X+2} & \mbox{ if } \mathbf{N}_{\mathbf{T}}(h) 
\le q^{\ell-X-1}\\ 0 & \mbox{ otherwise.} \end{cases}
$$
\end{Lemma}

Next we bound sums of the form
$$
U(X,Y;f,r):=\sum\limits_{\substack{\gamma\in \mathbb{A}\setminus \{0\} \\ 
\mathbf{N}(\gamma)\le q^Y}} \left|\Sigma\left(X,\frac{\gamma r}{f}\right)\right|.
$$
We prove the following.

\begin{Lemma} \label{Ubound}
Suppose that $X\ge \ell$, $Y\ge 0$, $r,f\in \mathbb{A}$, $\mathbf{N}(f)=q^Q$ and $(r,f)=1$. Then
$$
U(X,Y;f,r)\ll_{q,\ell} q^Q + q^Y.
$$
\end{Lemma}

\begin{proof}
By Lemma \ref{linexpsum}, we have 
$$
U(X,Y;f,r)= q^{X+2} \cdot \sharp\left\{\gamma\in \mathbb{A}\setminus \{0\} \ :\  \mathbf{N}(\gamma)\le q^Y, \ 
\mathbf{N}_{\mathbf{T}}\left(\frac{\gamma r}{f}\right)\le q^{\ell-X-1}\right\}.
$$
It follows that
\begin{equation*}
 U(X,Y;f,r)\le q^{X+2} \sum\limits_{\substack{s \in \mathbb{A}\setminus \{0\}\\ \mathbf{N}(s)\le q^{Q+\ell-X-1}}} 
 \sum\limits_{\substack{\gamma\in \mathbb{A} \\ \mathbf{N}(\gamma)\le q^Y\\ \gamma \equiv s\overline{r} \bmod{f}}} 1\\
 \ll_{\ell} q^Q\left(1+q^{Y-Q}\right)= q^Q + q^Y.
\end{equation*}
\end{proof}

Now assume that $\mathcal{S}=\mathbb{A}\setminus\{0)$. Then combining \eqref{sKnewbound} and Lemma \ref{Ubound}, we obtain
\begin{equation} \label{sKfullbound}
 \sK \ll_{q,\ell} 1+q^{2Q-N} \quad \mbox{if } M\ge 0.
\end{equation}
Further, combining \eqref{Tbound}, \eqref{M<0} and \eqref{sKfullbound} proves Theorem \ref{fullmodqe}. 

\section{Square moduli}
We now turn to the case of square moduli, i.e. the situation when
$$
\mathcal{S}:=\{f^2\ :\ f\in \mathbb{A}\setminus \{0\}\},
$$
and $\mbox{char}(K)=p>2$. Then we may write
\begin{equation} \label{specsquare}
\sum\limits_{f_2\in \mathcal{S}(\le Q)} 
e\left({\bf Tr}\left(\frac{jr_1f_2}{2f_1\sqrt{\alpha}}\right)\right) = 
2 \sum\limits_{\substack{f\in \mathbb{A}\\ 1\le \mathbf{N}(f)\le q^{Q/2}}} 
e\left({\bf Tr}\left(\frac{jr_1f^2}{2f_1\sqrt{\alpha}}\right)\right).
\end{equation}
To bound the above expression, we perform a Weyl shift. We bound the modulus square by 
\begin{equation} \label{weyl}
\begin{split}
& \Bigg|\sum\limits_{\substack{f\in \mathbb{A}\\ 1\le \mathbf{N}(f)\le q^{Q/2}}} 
e\left({\bf Tr}\left(\frac{jr_1f^2}{2f_1\sqrt{\alpha}}\right)\right) \Bigg|^2\\ 
= & \sum\limits_{\substack{f,\tilde{f}\in \mathbb{A}\\ 1\le \mathbf{N}(f)\le q^{Q/2}\\ 1\le \mathbf{N}(\tilde{f})\le q^{Q/2}}}
e\left({\bf Tr}\left(\frac{jr_1(\tilde{f}-f)(\tilde{f}+f)}{2f_1\sqrt{\alpha}}\right)\right)\\
= & \sum\limits_{\substack{f\in \mathbb{A}\\ 1\le \mathbf{N}(f)\le q^{Q/2}}} 1 +
\sum\limits_{\substack{f,h\in \mathbb{A}\\ 1\le \mathbf{N}(f)\le q^{Q/2}\\ 1\le \mathbf{N}(h)\le q^{Q/2}}}
e\left({\bf Tr}\left(\frac{jr_1h(h+2f)}{2f_1\sqrt{\alpha}}\right)\right)\\
\ll & q^{Q/2}+\sum\limits_{\substack{h\in \mathbb{A}\\ 1\le \mathbf{N}(h)\le q^{Q/2}}} \Bigg| 
\sum\limits_{\substack{f\in \mathbb{A}\\ 1\le \mathbf{N}(f)\le q^{Q/2}}}
e\left({\bf Tr}\left(\frac{jh(2r_1)f}{2f_1\sqrt{\alpha}}\right)\right) \Bigg|.
\end{split}
\end{equation}

Now assume $M\ge 0$. Then from \eqref{sKnewbound} and \eqref{specsquare}, we deduce that
\begin{equation*} 
\begin{split}
\sK 
\ll_{q,\ell} & q^{3Q/2-N}+ q^{Q+\ell-(N+2)}\max\limits_{\substack{r_1 \in \mathbb{A}, \ f_1 \in \mathcal{S}(Q) \\ (r_1,f_1)=1} }\
\sum_{\substack{j\in \mathbb{A}\\ 1 \le \mathbf{N}(j) \le q^{N+1-Q}}} \left| 
\sum\limits_{\substack{f\in \mathbb{A}\\ 1\le \mathbf{N}(f)\le q^{Q/2}}} 
e\left({\bf Tr}\left(\frac{jr_1f^2}{2f_1\sqrt{\alpha}}\right)\right)\right|.
\end{split}
\end{equation*}
where the term $q^{3Q/2-N}$ comes from the contribution of $j=0$. 
Further, applying the Cauchy-Schwarz inequality, and using \eqref{specsquare} and \eqref{weyl}, we obtain
\begin{equation*}
\begin{split}
\sK \ll_{q,\ell} &\ q^{3Q/2-N}+q^{Q+\ell-(N+2)}\max\limits_{\substack{r_1 \in \mathbb{A}, \ f_1 \in \mathcal{S}(Q) \\ (r_1,f_1)=1} }\
\Bigg(\sum_{\substack{j\in \mathbb{A}\\ 1\le \mathbf{N}(j) \le q^{N+1-Q}}} 1 \Bigg)^{1/2}\times\\ &
\left(\sum_{\substack{j\in \mathbb{A}\\ 1 \le \mathbf{N}(j) \le q^{N+1-Q}}} \left| 
\sum\limits_{\substack{f\in \mathbb{A}\\ 1\le \mathbf{N}(f)\le q^{Q/2}}} 
e\left({\bf Tr}\left(\frac{jr_1f^2}{2f_1\sqrt{\alpha}}\right)\right)\right|^2\right)^{1/2}\\ 
\ll_{q,\ell} &\ q^{3Q/2-N}+q^{(Q-N)/2} \max\limits_{\substack{r_1 \in \mathbb{A}, \ f_1 \in \mathcal{S}(Q) \\ (r_1,f_1)=1} }\
\Bigg(\sum_{\substack{j\in \mathbb{A}\\ 1\le \mathbf{N}(j) \le q^{N+1-Q}}} \Bigg(
q^{Q/2}+\\ & \sum\limits_{\substack{h\in \mathbb{A}\\ 1\le \mathbf{N}(h)\le q^{Q/2}}} \Bigg| 
\sum\limits_{\substack{f\in \mathbb{A}\\ 1\le \mathbf{N}(f)\le q^{Q/2}}}
e\left({\bf Tr}\left(\frac{jh(2r_1)f}{2f_1\sqrt{\alpha}}\right)\right)\Bigg|  \Bigg)\Bigg)^{1/2},
\end{split}
\end{equation*}
which implies
\begin{equation*}
\begin{split}
\sK \ll & q^{3Q/2-N}+q^{Q/4} + q^{(Q-N)/2}\max\limits_{\substack{r_1 \in \mathbb{A}, \ f_1 \in \mathcal{S}(Q) \\ (r_1,f_1)=1} }\\ &
\Bigg(\sum_{\substack{\ell\in \mathbb{A}\\ 1\le \mathbf{N}(\ell) \le q^{N+1-Q/2}}} \tau(\ell)
\Bigg| 
\sum\limits_{\substack{f\in \mathbb{A}\\ 1\le \mathbf{N}(f)\le q^{Q/2}}}
e\left({\bf Tr}\left(\frac{\ell(2r_1)f}{2f_1\sqrt{\alpha}}\right)\right)\Bigg|  \Bigg)^{1/2},
\end{split}
\end{equation*}
where $\tau(\ell)$ is the number of divisors of $\ell$ in $\mathbb{A}$. We note that
$$
\tau(\ell)\ll_q 2^{\mathbf{D}(\ell)}
$$
since $\ell$ factors into at most $\mathbf{D}(\ell)$ irreducible polynomials, each unique up to units, and each divisor of $\ell$ 
is associated to a product of a subset of these polynomials. 
Using this bound and Lemma \ref{Ubound}, we deduce that
\begin{equation} \label{sKsquarebound}
\begin{split}
\sK \le_{q,\ell} & q^{3Q/2-N} + q^{Q/4} + q^{(Q-N)/2}\left(2^{N-Q/2}\left(q^{N-Q/2}+q^Q\right)\right)^{1/2}
\\ \ll & q^{3Q/2-N}+2^{N/2-Q/4}\left(q^{Q/4}+q^{Q-N/2}\right) \quad \mbox{if } M\ge 0.
\end{split}
\end{equation}

Combining \eqref{Tbound}, \eqref{M<0} and \eqref{sKsquarebound}, we obtain
\begin{equation*}
\begin{split}
& \sum\limits_{\substack{f\in \mathbb{A}\setminus\{0\}\\ \mathbf{N}(f)= q^Q\\ f=\Box}} \ \sum\limits_{\substack{r \bmod{f}\\ (r,f)=1}} 
\left|\sum \limits_{g \in \mathbf{B}(0, N)\cap \mathbb{A}}  a_g \cdot e\left(\mathbf{Tr}\left(\frac{gr}{2\sqrt{\alpha}f}\right)\right)\right|^2 
\\ \ll & \left(q^{3Q/2}+q^N+2^{N/2-Q/4}\left(q^{Q/4+N}+q^{Q+N/2}\right)\right) \sum \limits_{g \in \mathbf{B}(0, N)\cap \mathbb{A}}   |a_g|^2,
\end{split}
\end{equation*}
where $f=\Box$ indicates that $f$ is a square in $\mathbb{A}$. This implies Theorem \ref{squaremodqe} upon changing $Q$ into $2Q$.

\end{document}